\documentclass{article}

\usepackage{mytemplate}
\usepackage{stix}
\usepackage[alphabetic]{amsrefs}

\newtheorem{thm}{Theorem}
\newtheorem{cor}{Corollary}[section]
\newtheorem{prop}[cor]{Proposition}
\newtheorem{lem}[cor]{Lemma}

\newcommand{\pns}{{\P^{n*}}}
\newcommand{\pet}{\pi_1^{\acute et}}
\newcommand{\fqb}{{\overline{\F}_q}}

\newcommand{\Con}{\mathrm{Con}}
\newcommand{\comp}{\mathrm{comp}}

\newcommand{\fq}{\F_q}
\newcommand{\wreath}{\wr}
\renewcommand{\Im}{\mathrm{im}}

\renewcommand{\k}{{\mathscr{k}}}

\begin{document}

\numberwithin{equation}{section}

\title{Monodromy of Hyperplane Sections of Curves and Decomposition Statistics over Finite Fields}
\author{Alexei Entin}
\maketitle

\abstract{For a projective curve $C\ss\P^n$ defined over $\F_q$ we study the statistics of the $\F_q$-structure of a section of $C$ by a random hyperplane defined over $\F_q$ in the $q\to\ity$ limit. We obtain a very general equidistribution result for this problem. We deduce many old and new results about decomposition statistics over finite fields in this limit. Our main tool will be the calculation of the monodromy of transversal hyperplane sections of a projective curve.}

\section{Introduction} Let $\F_q$ be a finite field, $C\ss\P^n$ a projective curve (by which we mean a closed one-dimensional subvariety) defined over $\F_q$, possibly singular and reducible.
Let $d=\deg C$ be the degree of the curve $C$ in $\P^n$. Let $C_1,\ldots,C_m$ be the $\F_q$-irreducible components of $C$ and $d_i=\deg C_i$ (we have $d=\sum d_i$).

Denote by $\pns$ the dual projective space of hyperplanes in $\P^n$. Consider the open subset
\beq\label{defsec}\mathrm{Sec}(C)=\left\{H\in\pns:\abs{H\cap C}=d\right\}\ss\pns\eeq parametrizing transversal hyperplane sections of $C$. We will view varieties as sets of points over $\fqb$ (the algebraic closure of $\F_q$) and for a variety $X$ defined over $\F_q$ we denote by $X(\F_q)$ its set of $\F_q$-points and by $\Fr_q$ the Frobenius map acting on $X$ (thus $X(\F_q)$ is the set of fixed points of $\Fr_q$). If we take $H\in\Sec(C)(\F_q)$ then the set $H\cap C$ is preserved by the action of $\Fr_q$. In fact $\Fr_q$ permutes each $H\cap C_i$ and since $|H\cap C_i|=d_i$ its cycle structure on each $H\cap C_i$ defines a conjugacy class in the group
$S_{d_1}\times\ldots\times S_{d_m}$ (product of the symmetric groups of degree $d_1,\ldots,d_m$), which we denote by $\Fr(H\cap C)$.

To simplify notation we will denote $$S_{d_1,\ldots,d_m}=S_{d_1}\times\ldots\times S_{d_m}.$$ We will study the distribution of $\Fr(C\cap H)$ in the set of conjugacy classes in $S_{d_1,\ldots,d_m}$ as $H$ varies over $\Sec(C)(\F_q)$. We will do so in the regime where $n,d$ are fixed and $q\to\ity$. Our main result is the following

\begin{thm}\label{thmdec} Let $C\ss\P^n$ be a quasireflexive curve defined over $\F_q$ with components $C_1,\ldots,C_m$ which we assume to be absolutely irreducible of degree $d_1,\ldots,d_m$. Denote $d=\deg C=\sum d_i$. Let $\mathcal{C}$ be a conjugacy class in $S_{d_1,\ldots,d_m}$. Then we have
$$\left|\{H\in\Sec(C)(\F_q)|\Fr(C\cap H)=\mathcal{C}\}\right|=\frac{|\mathcal{C}|}{\left|S_{d_1,\ldots,d_m}\right|}q^n\lb 1+O_{n,d}\lb q^{-1/2}\rb\rb.$$
\end{thm}

The notion of a quasireflexive curve will be defined in section \ref{secref}, it is a slight generalization of the notion of a reflexive curve which will be recalled below.
Note that under the conditions of the theorem $\left|\Sec(C)(\F_q)\right|=q^n+O_{n,d}\lb q^{n-1}\rb$ since $\Sec(C)$ is an open subset of $\pns$ defined by the nonvanishing of a certain polynomial with degree bounded in terms of $n,d$ (see section \ref{seccheb}), so essentially the theorem asserts the equidistribution of the Frobenius classes of $C\cap H$ in the space of conjugacy classes of $S_{d_1}\times\ldots\times S_{d_m}$.

{\bf Remark 1.} The condition of absolute irreducibility for the $C_i$ in the statement of the theorem is not essential, but the statement needs to be slightly modified without it. See Theorem \ref{cormain} for the precise statement in the more general case.

{\bf Remark 2.} The assertion of Theorem \ref{thmdec} for smooth irreducible plane curves was proved by Bary-Soroker and Jarden \cite{BaJa12} and a slightly weaker assertion (but essentially equivalent) for irreducible plane curves was established in a recent work by Makhul, Schicho and Gallet \cite{MSG_1}.

We will present several applications of Theorem \ref{thmdec}. We give a unified treatment and slight improvements to several results due to Bank, Bary-Soroker, Carmon, Entin, Foster, Jarden, Rudnick and others on decomposition statistics in function fields in the $q\to\ity$ regime (Corrolaries \ref{bh} and \ref{corg}). We will also compute the distribution of the $\F_q$-structure of the intersection of $n$ hypersurfaces in $\P^n$ defined over $\F_q$ (Corollary \ref{corgal}), again in the $q\to\ity$ limit. We will also apply our main geometric result (Theorem \ref{thm1}) to the problem of computing the Galois group of a polynomial with indeterminate coefficients (Corollary \ref{corgal}), generalizing results of Uchida, Smith and Cohen \cite{Coh80}.

To prove Theorem \ref{thmdec} we will need to compute the monodromy of hyperplane sections of a projective curve. As a first step we will need to understand the geometric situation over an algebraically closed field. Let $\k$ be an algebraically closed field and $C\ss\P^n$ a projective curve defined over $\k$. Consider the variety
\beq\label{defpsec}\PSec(C)=\{(H,P)|H\in\Sec(C),P\in C\cap H\}\ss \pns\times\P^n\eeq parametrizing transversal hyperplane sections with a chosen point on the section (pointed transversal sections). The projection map $\phi:\PSec(C)\to\mathrm{Sec}(C)$ is finite \'etale and therefore we may consider the monodromy action of the \'etale fundamental group $\pet(\Sec(C))$ on a fiber of $\phi$. Assume that $C$ has irreducible components $C_1,\ldots,C_m$ of degree $d_1,\ldots,d_m$ respectively. Then a fiber $\phi^{-1}(H)$ for $H\in\Sec(C)$ is in a natural bijection with $C\cap H=\cup_{i=1}^m(C_i\cap H)$ and $\pet(\Sec(C),H)$ acts on each subset $C_i\cap H$ which has cardinality $d_i$, so we obtain a homomorphism $\pet(\Sec(C),H)\to S_{d_1}\times\ldots\times S_{d_m}$, which is defined up to conjugation. We call its image the monodromy of hyperplane sections of $C$ and denote it by $\Mon(\PSec(C)/\Sec(C))$. It is a subgroup of $S_{d_1,\ldots,d_m}$ well-defined and independent of $H$ up to conjugation.

For an irreducible curve $C$ in characteristic zero this monodromy group was studied by Harris \cite{Har81} and shown to be all of $S_d$. If $\k=\C$ we may replace the \'etale fundamental group with the topological fundamental group and the resulting monodromy will be the same. In characteristic $p>0$ it is not always true that the monodromy is all of $S_d$. It was shown by Rathmann \cite{Rat87} and Ballico and Hefez \cite{BaHe86} that if $C$ is irreducible and reflexive its monodromy is $S_d$ (Ballico and Hefez also extended this to higher dimensional varieties). A curve $C\ss\P^n$ is said to be reflexive if the map from its conormal variety to its dual variety is birational. We will recall the precise definitions in the next section. The condition of reflexivity is always satisfied in characteristic 0 and never in characteristic 2. It is usually satisfied in odd characteristic, but there are rare counterexamples. In the next section we will define the notion of a \emph{quasireflexive} curve in characteristic 2. This is a technical condition which is usually, but not always, satisfied.

We generalize the above result of Rathmann-Ballico-Hefez to reducible curves. This generalization can be of independent interest to algebraic geometers. For example it can be used to study general and uniform position properties for hyperplane sections of reducible curves and try to deduce Castelnuovo-type bounds for reducible curves. This was undertaken by Ballico by a different method.

\begin{thm}\label{thm1} Let $C\ss\P^n$ be a curve with components $C_1,\ldots,C_m$ of degree\\
 $d_1,\ldots,d_m$. Assume that $C$ is reflexive or that $\chr\k=2$ and $C$ is quasireflexive. Then
$$\Mon(\PSec(C)/\Sec(C))=S_{d_1}\times\ldots\times S_{d_m}.$$\end{thm}

To prove Theorem \ref{thmdec} we will first compute the geometric monodromy of hyperplane sections of $C$ (i.e. over $\fqb$) via Theorem \ref{thm1}, then derive the arithmetic monodromy over $\F_q$ (Proposition \ref{thmmon}) and finally apply a Chebotarev density theorem (Theorem \ref{thmcheb}) to recover the distribution of $\Fr(H\cap C)$.

We now list our applications, which will be derived in detail in section \ref{secapp}. For a squarefree polynomial $f\in\F_q[t],\deg f=d$ we define its Frobenius class $\Fr(f)$ to be the class in $S_d$ of the Frobenius action on its $d$ roots. It corresponds precisely to the set of degrees appearing in the decomposition of $f$ into irreducibles (over $\F_q$). For squarefree $f_1,\ldots,f_m\in\F_q[t],\deg f_i=d_i$ we define a Frobenius class $\Fr(f_1,\ldots,f_m)\in S_{d_1,\ldots,d_m}$ by taking the product of $\Fr(f_i)$.

\begin{cor}\label{bh} Let $F_1,\ldots,F_m\in\F_q[t,x]\sm\F_q[t,x^p]$ be absolutely irreducible non-associate polynomials, $n$ a natural number. Let $$d_i=\deg_t F_i\lb t,A_0+A_1t+\ldots+A_nt^n\rb,$$ where the $A_i$ are independent variables and $d=\sum d_i$. Denote
$$U_n(\F_q)=\left\{f\in\F_q[t],\deg f=n\mid F_i(t,f(t))\mbox{ squarefree of degree }d_i\right\}.$$ Assume that one of the following holds:
\begin{enumerate}
\item[(i)] $n\ge 3$.
\item[(ii)] $q$ is odd and $n\ge 2$.
\item[(iii)] $\chr\F_q>\max d_i$ and $n\ge 1$.
\end{enumerate}
Let $\mathcal{C}$ be a conjugacy class in $S_{d_1,\ldots,d_m}$.
Then
\begin{multline*}\left|\left\{f\in U_n(\F_q)\mid\Fr\lb F_1(t,f_1(t)),\ldots,F_m(t,f(t))\rb=\mathcal{C}\right\}\right|=\\
=\frac{|\mathcal{C}|}{\left|S_{d_1,\ldots,d_m}\right|}q^{n+1}\lb 1+O_{d}\lb q^{-1/2}\rb\rb.\end{multline*}
In particular the number of $f\in\F_q[t],\deg f=n$ such that all the values \\$F_i(t,f(t))\in\F_q[t],1\le i\le m$ are irreducible is $$\frac{q^{n+1}}{d_1\ldots d_m}\lb 1+O_d\lb q^{-1/2}\rb\rb.$$
\end{cor}

We note that the final assertion of the corollary is a function field analogue (in the $q\to\ity$ limit) of the classical Bateman-Horn conjecture on the frequency of prime values of polynomials \cite{BaHo62}. Corollary \ref{bh} was proved by the author in \cite{Ent16} with some additional restrictions on the $F_i$ and $n$ and in some of the cases the proof made use of the classification of finite simple groups. Here we remove the restrictions and eliminate the use of any nontrivial group theory. The corollary was also obtained independently by Carmon by a different method (unpublished note).

Next we obtain another corollary which slightly improves the result of Bank and Foster \cite{BaFo_1} on the decomposition statistics of divisors on curves. The setting is as follows: let $C/\F_q$ be a smooth irreducible projective curve of genus $g$, $E$ a divisor on $C$ defined over $\F_q$.
Let $f\in\F_q(C)$ be a rational function. Define the linear system $I(f,E)=\{(f+g)_0|g\in L(E)\}$, where $(h)_0$ denotes the divisor of zeros of a rational function $h$ and $L(E)$ is the Riemann-Roch space of $E$. Denote
$d=\deg\mathrm{lcm}((f)_\ity,E_\ity)$ ($D_\ity$ denotes the divisor of poles of a divisor or rational function). If we assume additionally that $\deg E\ge 2g$ then (by Riemann-Roch) the generic element $D$ of $I(f,E)$ is the sum of $d$ distinct $\fqb$-points on $C$ with an action of $\Fr_q$ permuting these points. This defines a Frobenius class in $S_d$ which we denote $\Fr(D)$. We denote by $I(f,E)'$ the subset of $D\in I(f,E)$ such that $D$ is squarefree.

\begin{cor}\label{corg} Let $C,E,f,d$ be as above and assume $\deg E\ge 2g+2$ if $q$ is odd and $\deg E\ge 2g+3$ if $q$ is even. Let $\mathcal{C}$ be a conjugacy class in $S_d$. Then we have
$$\card{ \left\{D\in I(f,E)'(\F_q)|\Fr(D)=\mathcal{C}\right\} }=
\frac{|\mathcal{C}|}{|S_d|}q^{\dim I(f,E)}\lb 1+O_{d}\lb q^{-1/2}\rb \rb.$$
In particular the probability that $D$ is irreducible over $\F_q$ is $1/d$ (up to an error of $O_d\lb q^{-1/2}\rb$).
\end{cor}

In \cite{BaFo_1} Bank and Foster prove a similar statement but with more restrictions on $E$ and $f$, in particular they require $\deg E\ge 6g+3$ in the odd $q$ case. In \cite{BaFo_2} they study the more general problem of correlations of decomposition of divisors (analogue of the Hardy-Littlewood conjecture for number fields). We will improve the result of \cite{BaFo_2} as well (Proposition \ref{prophl}).

Another application of Theorem \ref{thmdec} is to study the statistics of the $\F_q$-structure of the intersection of $n$ random hypersurfaces in $\P^n$ defined over $\F_q$. Let $n,d_1,\ldots,d_n$ be natural numbers and let $U_{d_1,\ldots,d_n}$ be the set of tuples $$\tau=(F_1,\ldots,F_n)\in\fqb[x_0,\ldots,x_n]^n$$ of homogeneous polynomials such that $\deg F_i=d_i$ and the hypersurfaces $H_i$ defined by $F_i=0$ intersect at $d=d_1\cdots d_n$ distinct points. This set is naturally a quasiprojective variety. If we also assume that $\tau\in U_{d_1,\ldots,d_n}(\F_q)$ (i.e. the $F_i$ have coefficients in $\F_q$) we have a Frobenius action on $H_1\cap\ldots\cap H_n$, defining a Frobenius class in $S_d$ which we denote $\Fr(\tau)$.

\begin{cor}\label{corint} Let $\mathcal{C}$ be a conjugacy class in $S_d$. then
$$\card{ \{\tau\in U_{d_1,\ldots,d_n}(\F_q)|\Fr(\tau)=\mathcal{C}\} }=\frac{|\mathcal{C}|}{|S_d|}q^{\dim U_{d_1,\ldots,d_n}}\lb 1+O_{n,d}\lb q^{-1/2}\rb\rb.$$
\end{cor}

Finally we give a direct application of Theorem \ref{thm1} to computing Galois groups of polynomials with coefficients depending on free variables.

\begin{cor}\label{corgal} Let $\k$ be any field with $\chr\k\neq 2$, $f_0,\ldots,f_n\in\k[t],n\ge 2$ polynomials with $\gcd(f_0,\ldots,f_n)=1$ and $A_1,\ldots,A_n$ free variables. Denote $K=\k(A_1,\ldots,A_n)$ and let $L$ be the splitting field of $$F(t)=f_0(t)+\sum_{i=1}^nA_if_i(t)\in K[t].$$
Assume that the rational functions $f_i/f_j,0\le i,j\le n$ generate $\k(t)$ over $\k$ and
that for some $i,j,k$ the Wronskian $W(f_i,f_j,f_k)$ doesn't vanish. Then $\Gal(L/K)=S_d$, where $d=\max\deg f_i$.
\end{cor}

This implies as a special case the Theorem of J. H. Smith \cite{Smi77} that a trinomial of the form $t^n+At^m+B,0<m<n$ with $A,B$ free variables and $(m,n)=1$ has Galois group $S_n$ unless $p|mn(n-m)$ and also the more general theorem of S. D. Cohen \cite{Coh80}*{Theorem 1} concerning the Galois group of polynomials of the form $$f(t)+Ax^m+Bx^k,f\in\k[t]$$ with certain conditions on $f,m,k$ and $\chr\k$.

The paper is organized as follows: in the next section we review the notion of a reflexive curve and its basic properties and define the notion of quasireflexivity, which differs from reflexivity only in characteristic 2. In section \ref{secthm1} we will prove Theorem \ref{thm1}. In section \ref{seccheb} we will state and prove a Chebotarev density theorem for varieties over $\F_q$. While this theorem is essentially folklore and appears in various versions in the literature we include the statement and proof of the precise version we use. In section \ref{secdec} we prove Theorem \ref{thmdec} in slightly greater generality (without the absolute irreducibility assumption). Then in section \ref{secapp} we will present some applications of our main results including the corollaries listed above. Finally in section \ref{seclem} we will prove Lemma \ref{lemproj}, which is a technical statement about quasireflexivity needed in some of the applications.

{\bf Acknowledgments.} The author would like to thank Edoardo Ballico, Dan Carmon and Lior Bary-Soroker for some useful discussions during the work leading to the present paper. The author would also like to thank Ofir Gorodetsky, Zeev Rudnick, Efrat Bank and Kaloyan Slavov for their comments on previous versions of this paper.

Part of the work leading up to the present paper was conducted during the author's Szeg\"{o} Instructorship at Stanford University.

\section{Reflexivity and quasireflexivity of projective curves}
\label{secref}

In this section we recall the basic properties of reflexive curves that we will need. More details can be found in \cite{Kle85}.
We will also define the notion of \emph{quasireflexivity}, which is more useful in characteristic 2. Let $\k$ be an algebraically closed field.
Let $C\ss\P^n$ be a projective curve, by which we mean a closed subvariety of $\P^n$ of dimension 1 (possibly reducible and singular). We will assume throughout this section that $n\ge 2$.
The \emph{conormal variety} of $C$ is defined to be the Zariski closure of the set
$$\{(P,H)\in \P^n\times\pns| H\mbox{ is tangent to }C\mbox{ at a smooth point }P\}.$$
We denote the conormal variety by $\Con(C)$. It is a subvariety in $\P^n\times\pns$ of dimension $n-1$.

If no component of $C$ is a line then the image of the projection of $\Con(C)$ to $\pns$
is an $(n-1)$-dimensional variety called the \emph{dual variety} of $C$ and denoted $C^*$. The projection $\Con(C)\to C^*$ is generically finite.
The curve $C$ is called \emph{reflexive} if the map $\Con(C)\to C^*$ is birational. In this case if we form the dual of $C^*$ by the same
recipe we obtain $C$ itself. If $C$ has the irreducible components $C_1,\ldots,C_m$ then $\Con(C)=\cup_{i=1}^m\Con(C_i)$ and $C^*=\cup_{i=1}^m C_i^*$.
The curve $C$ is reflexive iff each $C_i$ is reflexive.

The \emph{Segre-Wallace criterion} \cite{Wal56} asserts that an irreducible curve $C$ is reflexive iff the field extension $\k(\Con(C))/\k(C^*)$ is separable. In particular in characteristic 0 every curve is reflexive. On the other hand in characteristic 2 no curve is reflexive \cite{Kat73}.
The \emph{Hefez-Klein generic order of contact theorem} \cite{HeKl85} asserts that for an irreducible non-reflexive curve $C$ which is not a line the following holds:
for a generic tangent $H$ to $C$ at a generic point $P$ the order of contact (multiplicity of intersection)
$I(P, H.C)$ equals the degree of inseparability of the extension $\k(\Con(C))/\k(C^*)$ (i.e. the degree of the largest purely inseparable subextension).

{\bf Definition.} We say that a curve $C\ss\P^n$ of degree $d$ is \emph{quasireflexive} if every component $C_i$ of $C$ has a tangent hyperplane $H$ such that
$|H\cap C|=d-1$, i.e. the tangency is as simple as possible and all other intersections are transversal. We call such an $H$ a \emph{simple tangent hyperplane}. If every component of $C$ has a simple tangent hyperplane then a generic tangent hyperplane to $C$ is simple. In the case of plane curves Bary-Soroker and Jarden \cite{BaJa12} called such curves "characteristic-0-like".

A point $P\in C$ is called a \emph{flex} (or inflection point) if the generic tangent hyperplane $H$ to $C$ at $P$ satisfies $I(P, H.C)>2$. We caution the reader that some authors define a flex differently in positive characteristic, but we stick to this classical terminology. The curve $C$ is quasireflexive iff a generic point $P\in C$ is not a flex and if a generic $H\in C^*$ is tangent at a single point. It follows from this that $C$ is quasireflexive iff each component $C_i$ is quasireflexive and $C_i^*\neq C_j^*$ for $i\neq j$.

\begin{prop}\label{propquasi} Assume $\chr\k\neq 2$. The following conditions are equivalent:
\begin{enumerate}\item[(i)] $C$ is quasireflexive.
\item[(ii)] $C$ is reflexive.
\item[(iii)] A generic point $P\in C$ is not a flex.
\end{enumerate}\end{prop}

\begin{proof} The generic order of contact theorem combined with the Segre-Wallace criterion imply (assumint $\chr\k\neq 2$) that (ii) is equivalent to (iii). By the paragraph preceding the proposition (i) implies (iii). It remains to show that if $C$ is reflexive it is quasireflexive. Since $\Con(C)\to C^*$ is birational it is generically one-to-one, so the generic tangent to $C$ is tangent at a single point. By the implication (ii)$\rightarrow$(iii) a generic $P\in C$ is not a flex, so $C$ is quasireflexive.
\end{proof}

A point $O\in\P^n$ is called a \emph{strange point} for the curve $C$ if every tangent line to one of the components $C_i$ contains $O$. Equivalently, every tangent hyperplane $H$ to $C_i$ contains $O$. The point $O$ is strange iff the projection $\P^n\sm\{O\}\to\P^{n-1}$ from $O$ is not generically \'etale on $C$. A curve is called \emph{strange} if it has a strange point. Note that the strange point is not required to lie on $C$.

The following lemma will be useful to establish quasiprojectivity for the curves appearing in our applications. We delegate the proof to section \ref{seclem}.

\begin{lem}\label{lemproj} Let $C\ss\P^n,n\ge 3$ be a curve and $O\in\P^n$ a non-strange point for $C$.
Let $\pi:\P^n\sm \{O\}\to\P^{n-1}$ denote the projection from $O$. Assume that the image $\pi(C)$ is quasireflexive. If $\chr\k=2$ assume additionally that no component of $\pi(C)$ is contained in a (two-dimensional) plane.
Then $C$ is quasireflexive.\end{lem}

\section{Proof of Theorem \ref{thm1}}\label{secthm1}

For the basic theory of the \'etale fundamental group and its monodromy action we refer the reader to \cite{Mil80}*{\S I.5} and \cite{Mur67}. Throughout this section we work over an arbitrary algebraically closed field $\k$. Let $\phi:X\to Y$ be a generically \'etale morphism of varieties and assume that $Y$ is irreducible. Over some open $U\ss Y$ the map $\phi$ is finite \'etale and $\pet(U)$ acts on the fiber $\phi^{-1}(y)$ over any point $y\in U$. This is the monodromy action, which is well defined and independent of $U,y$ up to conjugation in $\pet(U)$. We denote by $\Mon(X/Y)$ the monodromy group which is the quotient of $\pet(U)$ by the subgroup fixing one (and therefore every) fiber of $\phi$.
If $X$ has irreducible components $X_1,\ldots,X_m$ with $\deg\phi|_{X_i}=d_i$ then $\Mon(X/Y)$ can be viewed as a subgroup of $S_{d_1,\ldots,d_n}$, well-defined up to conjugation. The surjective map $\Mon(X/Y)\to\Mon(X_i/Y)$ is induced by the projection
$S_{d_1,\ldots,d_m}\to S_{d_i}$. We have
\beq\label{mongal}\Mon(X/Y)\cong\Gal(L/\k(Y))\eeq where $L$ is the Galois closure of the composite of $\k(X_i),1\le i\le m$ (for a variety $X$ we denote by $\k(X)$ its field of rational functions).

The following proposition generalizes \cite{BaHe86}*{Propositions 2,3} to the reducible case.

\begin{prop} Let $Y$ be an irreducible variety, $X$ a variety with irreducible components $X_1,\ldots,X_m$ and $\phi:X\to Y$ a generically \'etale morphism. Denote $\phi_i=\phi|_{X_i},d_i=\deg\phi_i$. Assume that
\begin{itemize}
\item[(i)] For each $i$ the component $(X_i\times_Y X_i)\sm\Delta_{X_i/Y}$ of the fiber product $X_i\times_Y X_i$ is irreducible ($\Delta_{X_i/Y}$ denotes the diagonal component in $X_i\times_Y X_i$).
\item[(ii)] For each $i$ there exists a smooth point $y\in Y$ such that $\phi_j$ is \'etale over $y$ for all $j\neq i$ and $\phi_i^{-1}(y)=\{x',x_1,\ldots,x_{d_i-2}\}$ such that $\phi_i$ is \'etale at $x_1,\ldots,x_{d_i-2}$ and $X_i$ is formally irreducible at $x'$ (i.e. the completion of its local ring is integral).
\end{itemize}
Then we have $\Mon(X/Y)=S_{d_1,\ldots,d_m}$.
\end{prop}

\begin{proof} By \cite{BaHe86}*{Proposition 2} condition (i) implies that the action of $\Mon(X_i/Y)$ on the respective fiber is doubly transitive. By \cite{BaHe86}*{Proposition 3} condition (ii) implies that each $\Mon(X_i/Y)$ contains a transposition. In fact the proof of that proposition can be adapted with insignificant changes to the reducible case, and under condition (ii) it implies a stronger fact: for each $i$, $\Mon(X/Y)$ contains an element acting by transposition on the fiber of $\phi_i$ and fixing the fibers of $\phi_j,j\neq i$. By double transitivity we see that $\Mon(X/Y)$ contains the full group $S_{d_i}\ss S_{d_1,\ldots,S_n}$ of permutations of a fiber of $\phi$ leaving the fibers of $\phi_j,j\neq i$ fixed. This holds for each $i$, so $\Mon(X/Y)=S_{d_1,\ldots,d_m}$.
\end{proof}

Now we are ready to prove Theorem \ref{thm1}. Let $C\ss\P^n$ be a quasireflexive curve with components $C_1,\ldots,C_m$. Denote $\deg C=d,\deg C_i=d_i$. We want to apply the last proposition to $$X=\{(H,P)\in\pns\times\P^n|P\in H\cap C\},Y=\pns$$ and the projection map $\phi:X\to \pns=Y$. Note that $X$ is a projective bundle over $C$ via the projection $(H,P)\mapsto P$. The projection map $\phi$ is generically \'etale, since over the open subset $\Sec(C)$ it restricts to the projection $\PSec(C)\to\Sec(C)$ (recall that these varieties are defined by (\ref{defsec}),(\ref{defpsec})). We have $\Mon(\PSec(C)/\Sec(C))=\Mon(X/Y)$. We denote $$X_i=\{(H,P)\in\pns\times\P^n|P\in H\cap C_i\},$$ these are the irreducible components of $X$.

In the proof of the main theorem of \cite{BaHe86} it is shown that $X_i\times_Y X_i\sm\Delta_{X_i/Y}$ is irreducible for any curve $C\ss\P^n$ and in fact for a variety of any dimension. We note that this does not require the quasireflexivity assumption. It remains to verify condition (ii) of the proposition. By the quasireflexivity assumption for each $i$ there is a hyperplane $H\in\pns=Y$ tangent to $C_i$, such that $\card{H\cap C}=d-1$. In this case the map $\phi$ is \'etale at all but one of the points of $X$ lying over $H$, namely the point $x=(H,P)$ where $P$ is the point of tangency of $H$ to $C_i$. Since $H$ is generic we may assume that $P$ is a smooth point on $C_i$ and since $X$ is a projective bundle over $C$ the point $x\in X$ is smooth and therefore formally irreducible.
Thus condition (ii) is satisfied and consequently we have $\Mon(X/Y)=S_{d_1,\ldots,d_m}$, which concludes the proof.

\section{Decomposition statistics and the Chebotarev density theorem}
\label{seccheb}

Let $\F_q$ be a finite field, $\phi:X\to Y$ a finite \'etale morphism of $\F_q$-varieties defined over $\F_q$, with $Y$ geometrically irreducible (i.e. irreducible over $\fqb$). Denote $d=\deg\phi$. For a point $y\in Y(\F_q)$ the fiber $\phi^{-1}(y)$ is a finite \'etale $\F_q$-scheme of order $d$. Geometrically this fiber can be described as a set of $d$ points over $\fqb$ with a Frobenius action, which determines a conjugacy class in $S_d$. For a fixed $d$ one may wish to study the distribution of this class as $y$ ranges over $Y(\F_q)$. The main tool for studying this distribution is a Chebotarev density theorem for varieties over $\F_q$, which will be stated below.

Let $\phi:X\to Y$ be as above and assume that $Y$ is normal. The \'etale fundamental group $\pet(Y)$ acts on a fiber over a geometric point, which is a set of $d$ geometric points. We denote by $\Mon(X/Y)\ss S_d$ the corresponding monodromy group, which is only well-defined up to conjugation. The \emph{geometric monodromy} is the group $\Mon^g(X/Y)=\Mon(X\times\fqb/Y\times\fqb)$ and is naturally a subgroup of $\Mon(X/Y)$, the latter also being referred to as the \emph{arithmetic monodromy}. There is a natural exact sequence
\beq\label{monseq}1\to\Mon^g(X/Y)\to\Mon(X/Y)\to\Gal(\F_{q^\nu}/\F_q)\to 1,\eeq
where $\F_{q^\nu}$ is the minimal field such that $\Mon(X\times \F_{q^\nu}/Y\times \F_{q^\nu})=\Mon^g(X/Y)$.

Let $x\in Y(\F_q)$ be a point. We have a map $$\Gal\lb\fqb/\F_q\rb\cong\pet(y)\to\Mon(\phi^{-1}(y)/y)\hookrightarrow\Mon(X/Y),$$ which is well-defined up to conjugation in $\Mon(X/Y)$. The conjugacy class of the image of the Frobenius $\Fr_q$ under this map is called the \emph{Frobenius class} of $\phi$ at the point $y$ and denoted $\Fr(y)$ ($\phi$ is implied). Its cycle structure corresponds to the $\F_q$-structure of the fiber $\phi^{-1}(y)$ as described above. The image of $\Fr(y)$ under the second map in (\ref{monseq}) always equals the Frobenius map $\Fr_q\in\Gal(\F_{q^\nu}/\F_q)$.
We want to study the distribution of $\Fr(y)$ in the set of conjugacy classes of $\Mon(X/Y)$ as $y$ varies over $Y(\F_q)$ for $q$ large and $V$ of bounded complexity, a notion that we define next.

Let $X\ss\P^n$ be a locally closed set defined over an algebraically closed field $\k$. We define the complexity of $X$ to be $\max(n,D)$ where $D$ is the minimal number such that $X$ can be defined by
$$X=\{x\in\P^n|F_1(x)=\ldots=F_k(x)=0,G_1(x)\cdots G_m(x)\neq 0\}$$ with $F_i,G_j\in\k[x_0,\ldots x_n],\deg F_i,\deg G_j\le D$. We denote the complexity of $X$ by $\comp(X)$. We define the complexity of a morphism of locally closed sets $\phi:X\ss\P^N\to Y\ss\P^n$ (denoted $\comp(\phi)$) to be the complexity of its graph, viewed as a locally closed set in $\P^N\times\P^n\ss\P^{(N+1)(n+1)-1}$ (Segre embedding). Most standard algebro-geometric constructions when performed on quasiprojective varieties and morphisms thereof of complexity $\le C$ yield varieties of complexity $O_C(1)$. Note that this bound is independent of the base field. This includes taking irreducible components, taking images and fibers of morphisms and the formation of fiber products as well conormal and dual varieties. The easiest way to show this is by using ultraproducts as in \cite{BGT11}*{Appendix A}, which gives an ineffective (but finite) bound. In principle for every specific construction it is possible to obtain an effective bound for the complexity of the result in terms of the complexity of the input by a constructive algebraic argument, but we will not pursue this here. When we work over a non-algebraically closed field we will define complexity via the algebraic closure.

The Lang-Weil bound \cite{LaWe54} asserts that for an irreducible variety $X/\F_q$ we have $$\abs{X(\F_q)}=q^{\dim X}\lb 1+O_{\comp(X)}\lb q^{-1/2}\rb\rb.$$ Now we state the uniform explicit Chebotarev density theorem for varieties over finite fields.

\begin{thm}\label{thmcheb} Let $\phi:X\to Y$ be a finite \'etale morphism of quasiprojective varieties $X\ss\P^n,Y\ss\P^M$ defined over $\F_q$, with $Y$ absolutely irreducible. Denote $d=\deg\phi$. Let $\mathcal{C}$ be a conjugacy class in $\Mon(X/Y)$ mapped to $\Fr_q$ in (\ref{monseq}) and $\nu$ the number appearing in (\ref{monseq}). Then
$$\abs{\{y\in Y(\F_q)|\Fr(y)=\mathcal{C}\}}=\frac{\nu|\mathcal{C}|}{\abs{\Mon(X/Y)}}q^{\dim Y}\lb 1+O_{\comp(\phi)}\lb q^{-1/2}\rb\rb$$
\end{thm}

A similar statement also appears in \cite{ABR15}*{Theorem A.4} (stated in the language of rings), however the statement and proof there are slightly inaccurate. It is asserted there that $\comp(X)$ is bounded in terms of $\comp(Y),d$, which is generally false. The theorem can also be deduced from the zero-dimensional case of the (uniform) Deligne-Katz equidistribution theorem \cite{KaSa99}*{Theorem 9.7.13}. We give the more elementary geometric proof here.

\begin{proof} We may assume that $Y$ is smooth, otherwise replace $Y$ with its smooth locus and restrict $X,\phi$ accordingly. This increases the complexity of $\phi$ by at most $O_{\comp(\phi)}(1)$. Observe that $d=O_{\comp(\phi)}(1)$. Define $$X_Y^{(d)}=\{(x_1,\ldots,x_d)\in X\times_Y\ldots\times_Y X|x_i\neq x_j\mbox{ for }i\neq j\}.$$ Since the $d$-fold fiber product of $X$ with itself over $Y$ has complexity \\ $O_{\comp(\phi)}(1)$ so does $X_Y^{(d)}$.
There is an induced \'etale map $X_Y^{(d)}\to Y$. Let $W$ be an $\F_q$-irreducible component of $X_Y^{(d)}$. Over $\fqb$ it decomposes into $\nu$ connected components $W_1,\ldots,W_\nu$, where $\nu$ is the number appearing in (\ref{monseq}) (Note that for \'etale covers of a smooth variety the connected components coincide with the irreducible components). The components $W_i$ are defined over $\F_{q^\nu}$ and the Frobenius map $\Fr_q$ permutes them cyclically. The function field $\F_q(W)$ is the Galois closure of the composite of the fields $\F_q(X_i)$ ($X_i$ denoting the $\F_q$-components of $X$), viewed as subfields of a common algebraic closure of $\F_q(Y)$. All other $\F_q$-irreducible components of $X_Y^{(d)}$ are isomorphic to $W$ because the group $S_d$ acting on $X_Y^{(d)}$ permutes its irreducible components, since it is transitive on a fiber over any $y\in Y$.
 There is an induced \'etale map $\psi:W\to Y$, making $W$ a Galois cover of $Y$ with Galois group $G\cong\Mon(X/Y)$. Note that $G$ acts on $W$ by automorphisms defined over $\F_q$. The subgroup of elements $g\in G$ such that $g(W_i)=W_i$ for some (and therefore all) $i$ is precisely the geometric monodromy $\Mon^g(X/Y)$.

We use the left exponential notation for the action of $G$ and $z\mapsto z^q$ for the action of Frobenius $\Fr_q$ on $\fqb$-points of varieties defined over $\F_q$. Denote by $G_1$ the preimage of $\Fr_q$ under the map $\Mon(X/Y)\to\Gal\lb\F_{q^\nu}/\F_q\rb$. Let $g\in G_1$ be an element. We will now use $g$ as a twisting map to construct a variety $W'/\F_q$ such that
$W'\times\fqb\cong W_1$ via Weil's descent theory. For an introduction to this subject see \cite{MilneAG}*{\S 16}.
Define the bijection $\Phi_g:W_1\lb\fqb\rb\to W_1\lb\fqb\rb$ by
$\Phi_g=\prescript{g^{-1}}{}z^q$.
Since $\Phi_g^m=\Fr_q^m$ for some $m$ (e.g. the order of $g$) the map $\Phi_g$ defines a descent datum for the variety $W_1/\fqb$ with respect to the field extension $\fqb/\fq$. Consequently there exist a variety $W'/\F_q$ such that $W'\times\fqb\cong W_1$ with the $\Fr_q$-action on $W'\lb\fqb\rb$ corresponding to the action of $\Phi_g$ on $W_1\lb\fqb\rb$. Since $W_1$ is absolutely irreducible, so is $W'$.  We also note that since over $\fqb$ the variety $W'$ is isomorphic to an irreducible component of $X_Y^{(d)}$ we have $\comp(W')=O_{\comp(\phi)}(1)$. Therefore by the Lang-Weil bound we have
\beq\label{wlw}\abs{W'(\F_q)}=q^{\dim Y}\lb 1+O_{\comp(\phi)}\lb q^{-1/2}\rb\rb.\eeq

Let $w\in W_1\lb\fqb\rb$ be a point. Viewing $w$ as a point on $X_Y^{(d)}$ we may write it as $w=(x_1,\ldots,x_d)$ with $x_i\in X,\phi(x_i)=y$. We want to determine when $w\in W_1\lb\fqb\rb=W'\lb\fqb\rb$ falls in $W'(\F_q)$. This happens iff $w$ is fixed by $\Phi_g$, which is equivalent to $w^q=\prescript{g}{}w$, or $$^g(x_1,\ldots,x_d)=\lb x_1^q,\ldots,x_d^q\rb.$$
On the other hand we have $g\in\Fr(y)$ iff there exists an $h_0\in G$ such that $\prescript{gh_0}{}w=^{h_0}w^q$ or equivalently $\prescript{h_0}{}w\in W'(\F_q)$. In this case we have
\begin{multline*}\left\{h\in G|^hw\in W'(\F_q)\right\}=\left\{\left. h\in G\vphantom{W_1\lb\fqb\rb}\right|\prescript{h}{}w\in W_1\lb\fqb\rb,\prescript{h^{-1}gh}{}w=w^q\right\}=\\=
\Mon^g(X/Y)\cap C_G(g)h_0,\end{multline*}
here $C_G(g)$ denotes the centralizer of $g$. Since by assumption $g\in C_G(g)$ is mapped to a generator of $\Gal\lb\F_{q^\nu}/\F_q\rb$ in (\ref{monseq}) we have $$|C_G(g)h_0\cap \Mon^g(X/Y)|=
\frac{|C_G(g)|}{\nu}=\frac{|G|}{\nu|\Fr(y)|}.$$ On the other hand if $g\not\in\Fr(y)$ the set $\{h\in G|^hw\in W'(\F_q)\}$ is empty. Denote by $\mathcal{C}$ the conjugacy class of $g$ in $G$. Recall that $\psi:W\to Y$ is the map induced from $\phi$. Now noting that the action of $G$ on $\psi^{-1}(y)$ is free we conclude that
$$\abs{\psi^{-1}(y)\cap W'(\F_q)}=\left[\begin{array}{ll}|G|/\nu|\mathcal{C}|, & \Fr(y)=\mathcal{C},\\ 0, & \Fr(y)\neq\mathcal{C}.\end{array}\right.$$
Summing over all $y\in Y(\F_q)$ and using (\ref{wlw}) we obtain
\begin{multline*}\abs{\{y\in Y(\F_q)|g\in\Fr(y)\}}=\frac{\nu|\mathcal{C}|}{|G|}\abs{W'(\F_q)}=\\=\frac{\nu|\mathcal{C}|}{|G|}q^{\dim Y}\lb 1+O_{\comp(\phi)}\lb q^{-1/2}\rb\rb,\end{multline*}
as required.
\end{proof}

\section{$\F_q$-structure of hyperplane sections: proof of Theorem \ref{thmdec}} \label{secdec}

Now let $C\ss\P^n$ be a projective curve defined over $\F_q$. Let $C_1,\ldots,C_m$ be its $\F_q$ components. Over $\fqb$ each $C_i$ decomposes further into $\nu_i$ components which we denote by $C_{ij},j\in\Z/\nu_i$. We may assume that $$C_{ij}=C_{i0}^{q^j},j\in\Z/\nu_i\cong\Gal\lb\F_{q^\nu}/\F_q\rb$$ (we use the exponential notation $X^q$ to denote $X^{\Fr_q}$ for a variety $X\ss\P^n$ defined over $\fqb$). Denote $d=\deg C,d_i=\deg C_{ij}$. We have $\deg C_i=d_i\nu_i$ and $d=\sum d_i\nu_i$. Let $H\in\pns(\F_q)$ be a hyperplane defined over $\F_q$ intersecting $C$ transversally at $d$ points. The Frobenius $\Fr_q$ acts on $H\cap C$ by a permutation. This permutation preserves each $H\cap C_i$ and maps each $H\cap C_{ij}$ to $H\cap C_{i(j+1)}$ (recall that the indices $j$ lie in $\Z/\nu_i$).

We recall the definition of the permutational wreath product: let $H$ be a group acting on a set $\Om$ and $G$ another group. For a finite set $A$ we denote by $\Sym(A)$ its group of permutations. Denote by $H^G$ the set of functions $G\to H$. Consider the group $$H\wreath_\Om G=\{\sig\in\Sym(\Om\times G): (x,u)\mapsto (\psi(u)x,gu)|g\in G,\psi\in H^G\}.$$
This is the \emph{permutational wreath product} of $H$ with $G$ with respect to $\Om$. If $H=S_n$ is a symmetric group we will omit $\Om$ from the notation and simply write $S_n\wreath G$, with the implied standard action on $\Om=\{1,\ldots,n\}$.

Since the action of $\Fr_q$ on $H\cap C_i$ maps $H\cap C_{ij}$ to $H\cap C_{i(j+1)}$ it falls in $S_{d_i}\wreath \Gal\lb\F_{q^\nu_i}\rb\cong S_{d_i}\wreath\Z/\nu_i\ss S_{d_i\nu_i}$. The corresponding conjugacy class in $S_{d_i}\wreath\Z/\nu_i$ is well-defined (independent of labeling). We denote
$$S_{d_1,\ldots,d_m}^{\nu_1,\ldots,\nu_m}=\prod_{i=1}^m\lb S_{d_i}\wreath\Z/\nu_i\rb.$$
The action of $\Fr_q$ on $H\cap C$ determines a conjugacy class $\Fr(H\cap C)$ of $S_{d_1,\ldots,d_m}^{\nu_1,\ldots,\nu_m}$. We want to study the distribution of this class as $H$ varies over $\Sec(C)(\F_q)$ for $n,d$ fixed and $q\to\ity$.
From section \ref{seccheb} we know that this distribution is determined by the monodromy group $\Mon(\PSec(C)/\Sec(C))$ with $\PSec(C),\Sec(C)$ viewed as $\F_q$-varieties.

Over $\fqb$ we may write $$\PSec(C)=\coprod_{i=1}^m\coprod_{j=1}^{\nu_i}\PSec'(C_{ij})$$ with $\PSec'(C_{ij})^q=\PSec'(C_{i(j+1)})$ (the prime indicates that we delete points not lying over $\Sec(C)$, which doesn't affect monodromy) and therefore we may view $\Mon(\PSec(C)/\Sec(C))$ as a subgroup $$\Mon(\PSec(C)/\Sec(C))\ss S_{d_1,\ldots,d_m}^{\nu_1,\ldots,\nu_m}\ss S_d.$$ Denote $\nu=\mathrm{lcm}(\nu_1,\ldots,\nu_m)$. By (\ref{monseq}) there is a map
$$\Mon(\PSec(C)/\Sec(C))\to\Gal\lb \F_{q^\nu}/\F_q\rb\cong\Z/\nu,$$ where $\nu=\mathrm{lcm}(\nu_1,\ldots,\nu_m)$. We also have the projection $S_{d_1,\ldots,d_m}^{\nu_1,\ldots,\nu_m}\to\prod_{i=1}^m \Z/\nu_i$ and the natural product of projections $\Del=\Z/\nu\to\prod_{i=1}^m \Z/\nu_i$.
The diagram
\beq\label{diagram1}\begin{tikzcd}\Mon(\PSec(C)/\Sec(C)) \arrow[r, twoheadrightarrow] \arrow[d, hookrightarrow] & \Z/\nu \arrow[d, hookrightarrow, "\Del"] \\
 S_{d_1,\ldots,d_m}^{\nu_1,\ldots,\nu_m}\arrow[r, twoheadrightarrow, "p"] & \prod_{i=1}^m\Z/\nu_i\end{tikzcd}\eeq
 is commutative. Here in $\Z/\nu_i\cong\Gal\lb\F_{q^\nu_i}/\F_q\rb,\Z/\nu\cong\Gal\lb\F_{q^\nu}\rb$ we identify 1 with $\Fr_q$. Let $p,\Del$ be the maps in diagram (\ref{diagram1}). We have
 $$\Mon(\PSec(C)/\Sec(C))\ss p^{-1}\lb\im\Del\rb.$$

\begin{prop}\label{thmmon} Assume that $C$ is quasireflexive. Then with notation as above
$\Mon(\PSec(C)/\Sec(C))=p^{-1}(\im\Del)$.\end{prop}

\begin{proof} We first compute the geometric monodromy $$\Mon^g(\PSec(C)/\Sec(C))=\Mon(\PSec(C\times\fqb)/\Sec(C\times\fqb)).$$
Since $C$ is quasireflexive and its geometric components are $C_{ij},1\le i\le m,j\in\Z/\nu_i,\deg C_{ij}=d_i$, by Theorem \ref{thm1} we have $$\Mon^g(\PSec(C)/\Sec(C))=\prod_{i=1}^m\prod_{j=1}^{\nu_i} S_{d_i}$$ and when viewed as a subgroup of $S_{d_1,\ldots,d_m}^{\nu_1,\ldots,\nu_m}$ it is precisely the kernel of the projection $p$. Therefore
$\ker(p)\ss\Mon(\PSec(C)/\Sec(C))$. Since by (\ref{diagram1}) we have $p\lb\Mon(\PSec(C)/\Sec(C))\rb=\im\Del$ we conclude that $$\Mon(\PSec(C)/\Sec(C))=p^{-1}\lb\im\Del\rb.$$
\end{proof}

\begin{thm}\label{cormain} Let $C$ be as in the last theorem and let $\mathcal{C}$ be a conjugacy class in $S_{d_1,\ldots,d_m}^{\nu_1,\ldots,\nu_m}$ mapped to $(1,\ldots,1)$ under the projection $p$ in (\ref{diagram1}). Then
$$\abs{\left\{H\in\Sec(C)(\F_q)|\Fr(H\cap C)\in\mathcal{C}\right\}}=\frac{|\mathcal{C}|}{\prod_{i=1}^m (d_i!)^{\nu_i}}q^n
\lb 1+O_{n,d}\lb q^{-1/2}\rb\rb$$
(here $d=\deg C$).
\end{thm}

\begin{proof} We note that $$\abs{p^{-1}\lb\im\Del\rb}=\nu|\ker(p)|=\nu\prod_{i=1}^m(d_i!)^{\nu_i}.$$
The theorem follows by combining Proposition \ref{thmmon} and Theorem \ref{thmcheb} once we verify that the projection map
$\phi:\PSec(C)\to\Sec(C)$ has complexity $O_{n,d}(1)$. The graph of $\phi$ is the locally closed set
$$\Gamma=\{(H,H,P)\in\pns\times\pns\times\P^n|P\in H\cap C, I(P,H.C)=1\}.$$
Since a curve of degree $d$ in $\P^n$ can always be defined by equations of degree $\le d$ (see \cite{BGT11}*{Theorem A.3})
and the condition $I(P,H.C)=1$ can be expressed as the vanishing of certain minors depending on the coefficients of these equations we conclude that $\comp(\phi)=\comp(\Gamma)=O_{n,d}(1)$ as required.
\end{proof}

\begin{proof}[Proof of Theorem \ref{thmdec}] This is just a special case of Theorem \ref{cormain} with $\nu_1=\ldots=\nu_m=1$.\end{proof}

\section{Applications}
\label{secapp}

In the present section we prove the corollaries listed in the introduction. We will see that they all follow quite easily from our main theorems. Lemma \ref{lemproj} will be an important tool for establishing quasireflexivity for the curves we consider in sections \ref{secbh} and \ref{secg}.

Throughout this section $\k$ will be an algebraically closed field, which we will explicitly specify to be $\fqb$ when necessary.

\subsection{The Bateman-Horn conjecture over $\F_q[t]$ for large $q$: proof of Corollary \ref{bh}}
\label{secbh}

We will in fact prove a slightly more general statement than Corollary \ref{bh}, namely we will drop the absolute irreducibility condition. Let $q$ be a prime power, $n$ a natural number and $F_1,\ldots,F_m\in\F_q[t,x]\sm\F_q[t,x^p]$ non-associate irreducible polynomials. We want to study the joint decomposition statistics of $F_i(t,f(t)),1\le i\le m$ as $f$ varies over all degree $n$ polynomials in $\F_q[t]$ and in particular determine how often they are all irreducible.  Let $F_{ij}\in\F_{q^{\nu_i}}[t,x],j\in\Z/\nu_i$ be the irreducible factors of $F_i$ over $\fqb$. We may assume that
$F_{i(j+1)}=F_{i0}^{\Fr_q^j}$.

Let $A_0,\ldots,A_n$ be free variables and denote
$$d_i=\deg F_{ij}\lb t,\sum_{j=0}^nA_jt^j\rb, d=\sum_{i=1}^m d_i\nu_i.$$
Let $\mathcal{C}$ be a conjugacy class in $S_{d_1\nu_1,\ldots,d_m\nu_m}$. Let $f\in\F_q[t],\deg f=n$ such that $F_i(t,f(t))$ is squarefree of degree $d_i\nu_i$. For such an $F$ the Frobenius action on the roots of $F_i(t,f_i(t))$ defines a permutation in $S_{d_i}^{\nu_i}\ss S_{d_i\nu_i}$, since $\Fr_q$ maps the roots of $F_{ij}(t,f(t))$ to the roots of $F_{i(j+1)}(t,f(t))$. Thus we obtain a well-defined conjugacy class
$\Th(f)\ss S_{d_1,\ldots, d_m}^{\nu_1,\ldots,\nu_m}$ (by taking the product over $1\le i\le m$).

Denote $$U_n=\{f\in\fqb[t]\mid\deg f=n,F_i(t,f(t))\mbox{ squarefree of degree }d_i,1\le i\le m\}.$$ This can be viewed as an open set in $\A^{n+1}$ (with the coefficients of $f$ as coordinates). The proof of the next proposition will show that it is a nonempty open subset (this is also shown in \cite{Rud14}).

\begin{prop}\label{propbh} Assume that one of the following holds:
\begin{enumerate}
\item[(i)]$n\ge 3$.
\item[(ii)]$n\ge 2$ and $\chr\F_q\neq 2$.
\item[(iii)] $\chr\F_q>\max d_i$.
\end{enumerate}
Let $\mathcal{C}$ be a conjugacy class in $S_{d_1,\ldots,d_m}^{\nu_1,\ldots,\nu_m}$ that is mapped to $(1,\ldots,1)$ by the projection to $\prod_{i=1}^m\Z/\nu_i$. Then
\begin{equation}\label{frpol}\left|\{f\in U(\F_q)\mid \Th(f)=\mathcal{C}\}\right|
=\frac{|\mathcal{C}|}{\prod_{i=1}^m(d_i!)^{\nu_i}}q^{n+1}\lb 1+O_{d}\lb q^{-1/2}\rb\rb.\end{equation}
In particular the number of $f\in\F_q[t],\deg f=n$ with all $F_i(t,f(t))$ irreducible is $$\lb\prod_{i=1}^m{d_i}\rb^{-1}q^{n+1}\lb 1+O_{d}\lb q^{-1/2}\rb\rb.$$
\end{prop}

Note that Corollary \ref{bh} is the special case $\nu_1=\ldots=\nu_m=1$.

The proof of the proposition will make use of the rational normal curve. Recall that the rational normal curve $Z_n\in\P^n$ is defined as the projective closure of the affine model
$\{(t,t^2,\ldots,t^n)|t\in\k\}\ss\mathbf{A}^n$ (see \cite{Har81}*{\S IV.3, ex. 3.4}). The curve $Z_N$ is irreducible of degree $n$.

\begin{lem}\label{propnormal} The curve $Z_n$ is quasireflexive for $n\ge 2$.\end{lem}

\begin{proof} In the present proof $\k$ is an arbitrary algebraically closed field, but we will apply the lemma with $\k=\fqb$. Working on the affine patch $\mathbf{A}^n\ss\P^n$ with coordinates $(x_1,\ldots,x_n)$ let $H$ be the hyperplane given by $a_1x_1+\ldots+a_nx_n+a_0=0$ with $a_i\in\k$. We assume that $a_n\neq 0$ (this is true generically). Then
$$H\cap Z_n=\left\{(t,t^2,\ldots,t^n):\sum_{i=0}^na_it^i=0\right\}.$$ We may choose $a_0,\ldots,a_n$ such that the polynomial $\sum_{i=0}^na_it^i$ has exactly $n-1$ roots. Then $|H\cap Z|=n-1=\deg(Z_n)-1$ and $Z_n$ is quasireflexive.\end{proof}

\begin{proof}[Proof of Proposition \ref{propbh}] Denote $F(t,x)=\prod_{i=1}^m F_i(t,x)$. By our assumptions $\partial F/\partial x\neq 0$.
Let $C$ be the closure in $\P^{n+1}$ of the set
$$\{(x,t,t^2,\ldots,t^n)\in\A^{n+1}|F(t,x)=0\}.$$ We denote by $x,t_1,\ldots,t_n$ the coordinates in $\A^n$. Note that the projection to the coordinates $x,t_1$ gives a birational map of $C$ onto the plane curve $F(t,x)=0$, therefore its components over $\F_q$ are the closures $C_i$ of the affine models
$$\{(x,t,t^2,\ldots,t^n)\in\A^{n+1}|F_i(t,x)=0\}$$
and over $\fqb$ each $C_i$ splits further into $\nu_i$ irreducible components $C_{ij},j\in\Z/\nu_i$ with affine models
$$\{(x,t,t^2,\ldots,t^n)\in\A^{n+1}|F_{ij}(t,x)=0\},$$ where $F_{ij}\in\fqb[t,x]$ are the irreducible factors of $F_i$.

Let $H\in\pns(\F_q)$ be a hyperplane which is not the hyperplane at infinity. It has an affine model with equation
$ax+\sum_{i=1}^n a_it_i+a_0=0$. For all but $O\lb q^n\rb$ such hyperplanes we may further assume that $a=1$ and $a_n\neq 0$.
In this case the $t_1$-coordinates of the points in $H\cap C_i$ correspond to the roots of $F_i(t,f(t))$ where $f(t)=\sum_{k=0}^na_kt^k$ (with multiplicities corresponding as well). We see that $$\deg C_{ij}=\deg F_{ij}=d_i,\deg C_i=d_i\nu_i,\deg C=\sum_{i=1}^md_i\nu_i=d.$$
If $H\in\Sec(C)$ then $f\in U_n(\F_q)$ (in particular $U_n\neq\emptyset$). Conversely all but $O_{d}\lb q^n\rb$ polynomials $f\in U_n(\F_q)$ arise in this way, since $\abs{\Sec(C)(\F_q)}=q^{n+1}+O_{d}\lb q^{n}\rb$.

For $H\in\Sec(C)(\F_q)$ as above the action of $\Fr_q$ on the roots of $F_i(t,f)$ and on $H\cap C_i$ corresponds. Thus the Frobenius class
$\Th(f)\ss S_{d_1,\ldots,d_m}^{\nu_1,\ldots,\nu_m}$ defined above coincides with the Frobenius class of $H$ viewed as a point on $\Sec(C)(\F_q)$ under the finite \'etale map $\PSec(C)\to\Sec(C)$, provided that we show that $\Mon(\PSec(C)/\Sec(C))=S_{d_1,\ldots,d_m}^{\nu_1,\ldots,\nu_m}$.
By Proposition \ref{thmmon} this would follow if we can show that $C$ is quasireflexive and then (\ref{frpol}) will follow from Theorem \ref{cormain}.
Note that for $p,\Del$ as in (\ref{diagram1}) we have $|p^{-1}\lb\im\Del\rb|=\prod_{i=1}^m(d_i!)^{\nu_i}$.

We project the curve $C$ from the point $(0:1:0:\ldots:0)$. On the affine patch $\A^{n+1}$ this projection acts by
$(x,t_1,\ldots,t_n)\mapsto (t_1,\ldots,t_n)$. Since $\partial F/\partial x\neq 0$ this projection is generically \'etale on $C$ and so $(0:1:0:\ldots:0)$ is not a strange point of $C$. The image of this projection is precisely the rational normal curve $Z_n$ and therefore quasireflexive. By Lemma \ref{lemproj} it follows that $C$ is quasireflexive for $n\ge 2$ unless $\chr{k}=2$ and $Z_n$ is contained in a plane, which only happens if $n=2$ and we assumed that this is not the case.

It remains to treat the case $n=1$ and $\chr\F_q>\max d_i$. In this case $C$ is the plane curve $F(t,x)=0$ with geometric components of degree $d_i$. But the degree of any non-reflexive plane curve over $\k$ which is not a line is at least $\chr\k$ \cite{Hef89}, so the components of $C$ are reflexive or lines. If $C_i$ over $\fqb$ is the union of $\nu_i$ lines then $\Fr(H\cap C)$ is always a $\nu_i$-cycle, so such components can be ignored. The remaining part of $C$ is reflexive and therefore quasireflexive.

The last assertion of the proposition (about the number of $f$ such that $F_i(t,f(t))$ is irreducible) follows from (\ref{frpol}) by showing that the number of
$$\sig=(\sig_1,\ldots,\sig_m)\in p^{-1}(\Im(\Del))\ss \prod_{i=1}^m S_{d_i}^{\nu_i}$$ ($p,\Del$ as in (\ref{diagram1})) such that $\sig_i$ are full (length $d_i\nu_i$) cycles is $\prod_{i=1}^md_i^{-1}(d_i!)^{\nu_i}$. This elementary combinatorial fact is proved in
\cite{Bar12}*{Lemma 5.1}.

\end{proof}

\subsection{Decomposition statistics of divisors on curves: \\ proof of Corollary \ref{corg}}
\label{secg}

In the present subsection we work over $\k=\fqb$. Let $C/\F_q$ be a smooth absolutely irreducible projective curve of genus $g$, $E$ a divisor on $C$ defined over $\F_q$ and $f\in\F_q(C)$ a rational function on $C$. Denote $d=\deg\mathrm{lcm}((f)_\ity,E_\ity)$. This is the degree of a generic divisor in $I(f,E)=\{(f+g)_0|g\in L(E)\}$. If we assume $\deg E\ge 2g-1$ then by Riemann-Roch we have $\dim L(E)=\deg E-g+1$ and therefore $|I(f,E)(\F_q)|=q^{\deg E-g+1}$. Let $$f_1,\ldots,f_n,n=\deg E-g+1$$ be an $\F_q$-basis for $L(E)$.

Consider the rational map $$\phi:C\to\P^n: \phi(P)= (f(P):f_1(P):\ldots:f_n(P)).$$ Since $C$ is smooth it can be extended to a morphism $C\to\P^n$. Denote $C'=\phi(C)$. It is an irreducible projective curve defined over $\F_q$. We have $\deg C=d$.

\begin{prop}\label{propg} Assume that $\deg E\ge 2g+2,\chr\k\neq 2$ or $\deg E\ge 2g+3$. Then $C'$ is quasireflexive.\end{prop}

\begin{proof}
Consider the projection $\pi:\P^n\to\P^{n-1}$ from the point $(1:0:\ldots:0)$. It acts as $\pi(x_0:\ldots:x_n)=(x_1:\ldots:x_n)$. Denote $\psi=\pi\circ\phi$. We have $\psi(P)=(f_1(P):\ldots:f_n(P))$. By \cite{Har81}*{Corollary IV.3.2} $\psi$ is an isomorphic embedding of $C$ into $\P^{n-1}$. Consequently the maps $C\to C'\to C''$ are birational. The projection $\pi$ is birational and in particular generically \'etale on $C'$ and therefore the point $O$ is not a strange point of $C$.

Next we show that $C''$ is quasireflexive, which by Lemma \ref{lemproj} would imply the same for $C'$ (if $\chr\k=2$ by our assumption $n-1\ge 3$ and $C''\ss\P^n$ is not contained in a hyperplane since $f_1,\ldots,f_n$ are linearly independent, so the conditions of the lemma are satisfied). Let $P\in C''$ be a point. A hyperplane $H=\{a_1x_1+\ldots+a_nx_n=0\}$ is tangent at $P$ iff
$\sum_{i=1}^n a_if_i\in L(E-2\psi^{-1}(P))$ and $I(P,H.C'')>2$ iff $\sum_{i=1}^n a_if_i\in L(E-3\psi^{-1}(P))$. By Riemann-Roch and assuming $\deg E\ge 2g+2$ we have
$$\dim L\lb E-3\psi^{-1}(P)\rb=n-3<n-2=\dim L\lb E-3\psi^{-1}(P)\rb$$ and therefore for any $P\in C''$ there exists a hyperplane $H$ such that $I(P,H\cap C'')=2$. If $\chr\k\neq 2$ this implies that $C''$ is reflexive and therefore quasireflexive (see section \ref{secref}).

Now assume $\chr\k=2$ and $\deg E\ge 2g+3$. We have shown that for a generic tangent hyperplane $H$ at a point $P$ we have $I(P,H.C)=2$, i.e. the generic point of $C''$ is not a flex. However we still need to show that the generic tangent hyperplane to $C''$ is only tangent at one point. Let $P,Q\in C''$ be two distinct points. By the reasoning above the set
$$\{H\in\pns|I(P,H.C''),I(Q,H.C'')>1\}$$ is a linear projective space in bijection with
$$\P L\lb E-2\psi^{-1}(P)\rb \cap \P L\lb E-2\psi^{-1}(Q)\rb=\P L\lb E-2\psi^{-1}(P)-2\psi^{-1}(Q)\rb$$
(for a vector space $V$ we denote by $\P V$ the corresponding projective space).

Under the assumption $\deg E\ge 2g+3$ by Riemann-Roch we have $$\dim\P L\lb E-2\psi^{-1}(P)-2\psi^{-1}(Q)\rb=n-5.$$ If $n=4$ then the set of $H$ tangent at both $P$ and $Q$ is empty for any $P\neq Q$, so every tangent is only tangent at one point. Therefore we assume $n\ge 5$.
Consider the variety
$$W=\{(H,P,Q)\in\pns\times C''\times C''|H\mbox{ tangent to }C''\mbox{ at }P,Q\}.$$
Since the fibers of the projection $V\to C''\times C''$ have dimension $n-5$ we have $\dim W=n-3$ and so its projection to $\pns$ has dimension $n-3$. Since the dual of $C''$ has dimension $n-2$ the generic tangent hyperplane to $C''$ does not lie in the projection of $V$, so it is only tangent at one point. This shows that $C''$ is quasireflexive and by Lemma \ref{lemproj} it follows that $C'$ is reflexive.

\end{proof}

\begin{proof}[Proof of Corollary \ref{corg}] Denote by $U\ss\pns$ the open set of hyperplanes $H\in\pns(\F_q)$ defined by an equation of the form $x_0+\sum_{i=1}^na_ix_i=0$. The map $H\mapsto \phi^{-1}(H)$ defines $\F_q$-isomorphisms
$U\to I(f,E)$ and $U\cap\Sec(C')\to I(f,E)'$ (recall that $I(f,E)'$ is the subset of squarefree divisors in $I(f,E)$). If $H\in (U\cap\Sec(C'))(\F_q)$ then $\Fr(H\cap C)=\Fr(\phi^{-1}(H))$ (viewed as conjugacy classes in $S_d$).

Under the assumptions of Corollary \ref{corg}, the last proposition combined with Theorem \ref{thmdec} implies that
$$\card{ \{D\in I(f,E)'|\Fr(D)=\mathcal{C}\} }=
\frac{|\mathcal{C}|}{|S_d|}q^{\dim I(f,E)}\lb 1+O_{d}\lb q^{-1/2}\rb \rb$$
for any conjugacy class $\mathcal{C}$. Note that $|I(f,E)'|=q^{\dim I(f,E)}\lb 1+O_{d}(q^{-1})\rb$, since it is in bijection with $(U\cap\Sec(C'))(\F_q)$, the complement of which is the set of $\F_q$-points of a union of a proper linear subspace and the dual variety of $C'$, which has complexity $O_d(1)$.
\end{proof}

We conclude this section with a generalization to the decomposition statistics of $m$ shifted divisors, improving the result of \cite{BaFo_2}. Let $h_1,\ldots,h_m\in\F_q(C)$ be distinct rational functions, $E$ a divisor on $C$ defined over $\F_q$.
Let $$d_i=\deg\mathrm{lcm}((h_i)_\ity,E_\ity)$$ be the generic degree of $g+h_i$ for $g\in L(E)$. We may study the distribution of
$\Th(g)=\lb\Fr((g+h_1)_0),\ldots,\Fr((g+h_m)_0)\rb$ viewed as a conjugacy class in $S_{d_1,\ldots,d_m}$ as $g$ ranges over $U(\F_q)$, where $$U=\{g\in L(E)|(g+h_i)_0\mbox{ is squarefree},1\le i\le m\}.$$

\begin{prop}\label{prophl} In the above setting assume that $\deg E\ge 2g+2,\chr\F_q\neq 2$ or $\deg E\ge 2g+3$. Let $\mathcal{C}$ be a conjugacy class of $S_{d_1,\ldots,d_m}$. Then
$$\card{ \left\{g\in U(\F_q)|\Th(g)=\mathcal{C}\right\} }=
\frac{|\mathcal{C}|}{d_1!\cdots d_m!}q^{\dim L(E)}\lb 1+O_{d}\lb q^{-1/2}\rb \rb,$$
where $d=\sum d_i$. In particular the probability that all $(g+f_i)_0$ are irreducible is
$\lb\prod_{i=1}^m d_i\rb^{-1}+O_d\lb q^{-1/2}\rb$.
\end{prop}

\begin{proof} Let $f_1,\ldots,f_n\in\F_q(C),n=\deg E-g+1$ be a basis for $L(E)$.
Consider the rational maps $$\phi_i:C\to\P^n:\phi_i(P)=(h_i(P):f_1(P):\ldots:f_n(P)),$$ which as before can be extended to morphisms. Denote
$C_i'=\phi_i(C)$. We claim that $C_i'\neq C_j'$ for $i\neq j$. Indeed assume that
\beq\label{equality}(h_i(P):f_1(P):\ldots:f_m(P))=(h_i(Q):f_1(Q):\ldots:f_m(Q)).\eeq
Since the map $$\psi:C\to \P^{n-1}:P\mapsto(f_1(P):\ldots:f_m(P))$$ is birational, for all but finitely many $P,Q$ the equality (\ref{equality}) implies $P=Q$ and therefore $h_i(P)=h_j(P)$ which can only happen for finitely many $P$. Therefore $C_i'\cap C_j'$ is finite.

Let $\pi$ be the projection from $(1:0:\ldots:0)$ and $C''=\psi(C)=\pi(C')$. In the proof of Proposition \ref{propg} we established that $C''$ is quasireflexive and therefore by Lemma \ref{lemproj} $C'$ is quasireflexive and is not contained in a plane. Now the proof can be completed similarly to the proof of Corollary \ref{corg} by using Theorem \ref{thmdec}.
\end{proof}

{\bf Remark.} The last proposition generalizes \cite{BaFo_2}*{Theorem A}. There it was required that $q$ is odd, that $\deg E\ge 6g+3$ and some additional restrictions on $E$ and $h_1,\ldots,h_m$.

\subsection{Decomposition statistics of intersections of hyperplanes: \\proof of Corollary \ref{corint}}

Let $\F_q$ be a finite field, $\k=\fqb$ and $n,d_1,\ldots,d_n$ natural numbers. Let $H_1,\ldots,H_n$ be hypersurfaces in $\P^n$ with $\deg H_i=d_i$. Generically the intersection $H_1\cap\ldots\cap H_n$ consists of $d=d_1\cdots d_n$ points. If $H_i$ are defined over $\F_q$ then the Frobenius acts on $H_1\cap\ldots\cap H_n$ and determines a conjugacy class in $S_d$ which we denote by $\Fr(H_1\cap\ldots \cap H_n)$. We would like to study the distribution of this class as $(H_1,\ldots,H_n)$ varies over all $n$-tuples of hypersurfaces of degree $d_1,\ldots,d_n$ defined over $\F_q$ for $d$ fixed and $q\to\ity$.

Denote by $\P^{n*}_s$ the space of hypersurfaces in $\P^n$ of degree $s$. It has a natural structure of a projective space of dimension $\lb \begin{array}{cc} n+s\\ n\end{array}\rb-1$ (corresponding to the space of homogeneous polynomials of degree $d$ in $n$ variables up to a multiplicative constant). By Bertini's theorem for generic
$(H_1,\ldots,H_{n-1})\in\P^{n*}_{d_1}\times\ldots\times\P^{n*}_{d_{n-1}}$ the intersection $C=H_1\cap\ldots\cap H_{n-1}$ is a smooth irreducible curve of degree $d_1\cdots d_{n-1}$, i.e. the set
\begin{multline*}V=\left\{(H_1,\ldots,H_{n-1})\in\P^{n*}_{d_1}\times\ldots\times\P^{n*}_{d_{n-1}}:\right. \\ \left. H_1\cap\ldots\cap H_{n-1}\mbox{ irreducible curve}\vphantom{\P^{n*}_{d_{n-1}}}\right\}\end{multline*}
is nonempty and open in $\P^{n*}_{d_1}\times\ldots\times\P^{n*}_{d_{n-1}}$. One can find equations defining the complement of $V$ of degree depending only on $n,d_1,\ldots,d_n$ (not on the base field), so $\comp(V)=O_{n,d}(1)$ and therefore all but an $O_{n,d}(q^{-1})$ fraction of $(H_1,\ldots,H_{n-1})\in\prod_{i=1}^{n-1}\P^{n*}_d(\F_q)$ lie in $V(\F_q)$. Consequently it would be enough to show that for a fixed absolutely irreducible $C\ss\P^n$ the Frobenius class $\Fr(C\cap H_n),H_n\in\P^{n*}_{d_n}(\F_q)$ is equidistributed in $S_d$ up to an error of $O_{n,d}\lb q^{-1/2}\rb$ to deduce the same for $\Fr(H_1\cap\ldots\cap H_n),(H_1,\ldots,H_n)\in U(\F_q)$, where
$$U=\left\{(H_1,\ldots,H_n)\in\prod_{i=1}^m\P^{n*}_{d_i}:\abs{H_1\cap\ldots\cap H_n}=d\right\}.$$

\begin{prop} Let $C\in\P^n$ be an absolutely irreducible curve of degree $\deg C=d$ defined over $\F_q$ and $e\ge 2$ a natural number. Let $\mathcal{C}$ be a conjugacy class of $S_{de}$. Denote by $\Sec_e(C)\ss\P^{n*}_e$ the open subset of $H\in\P^{n*}_e$ that intersect $C$ transversally. Then
\begin{multline*}\abs{\{H\in\Sec_e(C)(\F_q)|\Fr(H\cap C)=\mathcal{C}\}}=\\=\frac{|\mathcal{C}|}{(de)!}
q^{\dim\P^{n*}_d}\lb 1+O_{n,d,e}\lb q^{-1/2}\rb\rb.\end{multline*}
\end{prop}

\begin{proof}

We will deduce this from Theorem \ref{thmdec}. Let $$\phi:\P^n\to\P^n_e=\P^M,M=\lb \begin{array}{cc} n+e\\ n\end{array}\rb-1$$ be the $e$-fold embedding \cite{Har81}*{\S I.2, ex. 2.12}. The space $\P^{n*}_e$ can be viewed as the dual of $\P^n_e$ as hyperplanes $h\ss\P^n_e$ correspond by $h\mapsto \phi^{-1}(h)$ to degree $e$ hypersurfaces in $\P^n$. The curve $\phi(C)$ is absolutely irreducible of degree $de$. There is a bijection
$\phi^{-1}(h)\cap C\leftrightarrow h\cap\phi(C)$ that respects the Frobenius action. We see that the proposition is equivalent to the assertion of Theorem \ref{thmdec} for the curve $\phi(C)$, which follows if we can show that $\phi(C)$ is quasireflexive. This is equivalent to showing that there is a hypersurface $H\in\P^{n*}_e$ that is a simple tangent to $C$, i.e. tangent to $C$ at a single point $P$ with $I(P,C.H)=2$ and intersecting $C$ transversally at all other points.

To this end we define
$$B=\{H\in\P^{n*}_e|H\mbox{ is tangent to }C\},$$
$$B_{P,Q}=\{H\in\P^{n*}_e|H\mbox{ is tangent to }C\mbox{ at P,Q}\},P\neq Q,$$
$$B_{P,P}=\{H\in\P^{n*}_e|I(P,H.C)>2\}.$$
It is easy to show that for $e\ge 2$ and smooth $P,Q\in C$ we have $$\dim B=M-1,\dim B_{P,Q}=M-4,\dim B_{P,P}\le M-3.$$ Note that these assertions may fail if $e=1$, even generically, because a hyperplane is tangent to $C$ at $P$ iff it contains the tangent line at $P$, which may be a line of inflection or coplanar with another tangent line. Also a generic $H\in B$ doesn't contain any singularities of $C$. From this we deduce by the same argument we used to conclude the proof of Proposition \ref{propg} that the generic $H\in B$ is a simple tangent.

\end{proof}

Corollary \ref{corint} now follows from the proposition and the preceding discussion if we assume (without loss of generality) that $d_n\ge 2$. Note that we formulated Corollary \ref{corint} in terms of defining polynomials instead of hypersurfaces, but these formulations are equivalent.

\subsection{Galois groups of polynomials with indeterminate coefficients: proof of Corollary \ref{corgal}}

Let $\k$ be a an algebraically closed field, $f_0,\ldots,f_n\in\k[t],n\ge 2$ nonzero polynomials with
$\mathrm{gcd}(f_1,\ldots,f_n)=1$ and $A_1,\ldots, A_n$ free variables. Denote $K=\k[A_1,\ldots,A_n]$. Let $L$ be the splitting field of the polynomial \begin{equation}\label{deff} F(t)=f_0(t)+\sum_{i=1}^nA_if_i(t)\in K[t].\end{equation} If $\chr\k=p$ is finite we require $F(t)\not\in K[t^p]$. The polynomial $F(t)$ is irreducible over $K$ since it is linear in each $A_i$. Denote $d=\deg F=\max\deg f_i$. Let $L$ be the splitting field of the separable polynomial $F(t)$. We would like to determine when $\Gal(L/K)$ is the full symmetric group $S_d$.

At this point we add the requirement that $f_i/f_j,0\le i,j\le n$ generate the field $\k(t)$.
Consider the rational map $\phi:\P^1\to\P^n$ defined by $$\phi((1:t))\mapsto(f_0(t):\ldots: f_n(t)).$$ It can be continued to a morphism $\P^1\to\P^n$. Our assumption that $f_i/f_j$ generate $\k(t)$ is equivalent to $\phi:\P^1\to\phi(\P^1)$ being birational. Denote $C=\phi(\P^1)$. We have $\deg C=d$.

\begin{prop}\label{propgal1} Under the above assumptions $$\Gal(L/K)\cong\Mon(\PSec(C)/\Sec(C)).$$ If $C$ is quasireflexive then $\Gal(L/K)=S_d$.\end{prop}

\begin{proof} Let $x_0,\ldots,x_n$ and $a_0,\ldots,a_n$ be the coordinates in $\P^n,\pns$ respectively. Consider the rational functions $A_i=a_i/a_0\in\k\lb\pns\rb,1\le i\le n$. The functions $A_1,\ldots,A_n$ are algebraically independent so the abuse of notation is justified. The subset $\Sec(C)\ss\pns$ is open, so $\k(\Sec(C))=\k(A_1,\ldots,A_n)=K$.

Next, since $\phi$ is birational we see that a generic point $P=(x_0:\ldots:x_n)\in C$ can be specified by the unique $t\in\k$ such that $P=(f_0(t):\ldots:f_n(t))$. If $H=(a_0:\ldots:a_n)\in\pns$ is a hyperplane containing $P$ then $\sum_{i=0}^n a_ix_i=0$ which is generically equivalent to
$$f_0(t)+\sum_{i=1}^nA_i(H)f_i(t)=0.$$ We see that $\PSec(C)$ is birational to
$$\left\{(t,A_1,\ldots,A_n)\in\k^{n+1}:f_0(t)+\sum_{i=1}^nA_if_i(t)=0\right\}$$ and therefore
$\k(\PSec(C))=K(\al)$, where $\al$ is a root of $F(t)$ ($F$ is defined by (\ref{deff})). The map
$\PSec(C)\to\Sec(C)$ induces the finite extension of function fields $K(\al)/K$ and therefore by (\ref{mongal}) we have
$\Mon(\PSec(C)/\Sec(C))=\Gal(L/K)$, where $L/K$ is the Galois closure of $K(\al)/K$.

The last assertion follows from Theorem \ref{thm1}.
\end{proof}

{\bf Remark.} In the last proposition we only used Theorem \ref{thm1} with $m=1$, which is due to J. Rathmann \cite{Rat87}*{Proposition 2.1}.

\begin{prop}\label{propgal2} In the above setting assume that $\chr\k\neq 2$ and that for some $0\le i,j,k\le n$ we have $$W(f_i,f_j,f_k)=\det\left[\ba{ccc}f_i&f_j&f_k\\f_i'&f_j'&f_k'\\f_i''&f_j''&f_k''\ea\right]\neq 0.$$ Then $C=\phi(\P^1)$ is quasireflexive.\end{prop}

\begin{proof} We parametrize the affine patch of $\P^1$ with the parameter $t$. Consider once again the birational map
$\psi(t)=(f_0(t):\ldots:f_n(t))$ defined above and let $t$ be a point at which it is a local isomorphism.
Consider the vectors
\begin{eqnarray}\nonumber v_0(t)&=&(f_0(t),\ldots,f_n(t))\\
\nonumber v_1(t)&=&(f_0'(t),\ldots,f_n'(t))\\
\nonumber v_2(t)&=&(f_0''(t),\ldots,f_n''(t)).\end{eqnarray}
Our condition implies that for a generic $t$ they are linearly independent. For a hyperplane $H=(a_0:\ldots:a_n)$ the order of contact $I(\psi(t),H.C)$ is the order of vanishing of $\sum_{i=1}^na_if_i(t)$ at $t$. By our assumption that $v_0(t),v_1(t),v_2(t)$ are linearly independent there exist $a_0,\ldots,a_n$ such that
$$\sum_{i=0}^na_if_i(t)=0,\sum_{i=0}^na_if_i'(t)=0,\sum_{i=0}^na_if_i''(t)\neq 0$$
and since $\chr\k=2$ this means that the order of vanishing of $\sum_{i=1}^na_if_i(t)$ at $t$ is exactly 2, i.e. the hyperplane $H=(a_0:\ldots:a_n)$ satisfies $I(\psi(t),H.C)=2$, so $\psi(t)$ is not a flex. By Proposition \ref{propquasi} the curve $C$ is quasireflexive.
\end{proof}

\begin{proof}[Proof of Corollary \ref{corgal}]
First of all observe that we may assume $\k$ to be algebraically closed, since extending the base field can only shrink the Galois group. Now the corollary follows immediately from propositions \ref{propgal1} and \ref{propgal2}.\end{proof}

\section{Proof of Lemma \ref{lemproj}}
\label{seclem}

In the present section we work over an arbitrary algebraically closed field $\k$. Let $C\ss\P^n$ with $n\ge 3$ be a curve (possibly reducible and singular) and $O\in\P^n$ a point which is not a strange point of $C$. Denote by $\pi:\P^n\sm\{O\}\to\P^{n-1}$ the projection from $O$. Denote $C'=\pi(C)$. By our assumption on $O$ the map $\pi:C\to C'$ is generically \'etale. We make the following additional assumptions:
\begin{enumerate}\item[(i)] $C'$ is quasireflexive.
\item[(ii)] If $\chr\k=2$ no component of $C'$ is contained in a plane.
\end{enumerate}
Under these assumptions we wish to prove that $C$ is quasireflexive.

Denote by $C_1,\ldots,C_m$ the irreducible components of $C$, $C_i'=\pi(C_i)$. Note that some of the $C_i'$ may coincide.
For a smooth point $P$ on a curve in projective space denote by $L_P$ its tangent line. If $P\in C$ is such that $O\not\in L_P$ (equivalently $\pi$ is unramified at $P$) then $\pi(L_P)=L_{\pi(P)}$. A hyperplane $H$ is tangent to $C$ at $P$ iff $L_P\ss H$.

Let $H$ be a hyperplane such that $O\in H$. Then $\pi(H)\ss\P^{n-1}$ is a hyperplane. Conversely if $h\ss\P^{n-1}$ is a hyperplane then $\pi^{-1}(h)$ is a hyperplane containing $O$.

\begin{lem} Let $P\in C$ be a smooth point such that $O\not\in L_P$, $H$ a hyperplane such that $O\in H$. We have $$I(\pi(P),\pi(H).C')=I(P,H.C).$$\end{lem}
\begin{proof} Let $(x_0:\ldots:x_n)$ be the coordinates in $\P^n$ chosen so that $O=(1:0:\ldots:0)$ and consider the rational functions $f_j\in\k(C),1\le j\le n$ defined by $f_i=x_i/x_0$. The projection $\pi$ acts by $$\pi(x_0:\ldots:x_n)=(x_1:\ldots:x_n)$$ and we will use the variables $(x_1:\ldots:x_n)$ in $\P^{n-1}$. For $P\in C$ we have $\pi(P)=(f_1(P):\ldots:f_n(P))$. Since $O\in H$ the hyperplane $H$ has an equation of the form $\sum_{i=1}^n a_ix_i=0$. The hyperplane $\pi(H)$ has the same equation. Since $O\not\in L_P$ the map $\pi:C\to C'$ is unramified at $P$ and therefore we have
$$I(P,H.C)=\mathrm{ord}_P\sum_{i=1}^na_jf_j(P)=I(\pi(P),H.C').$$
\end{proof}

\begin{cor} A generic $P\in C$ is not a flex.\end{cor}

\begin{proof} For a generic $P\in C$ the point $\pi(C)\in C'$ is not a flex. Let $h$ be a tangent to $\pi(C)$ such that $I(\pi(P),h.C')=2$ and denote $H=\pi^{-1}(H)$. We have $I(P,H.C)=I(\pi(P),h.C')=2$, so $P$ is not a flex.\end{proof}

If $\chr\k\neq 2$ this already shows that $C$ is quasireflexive (by Proposition \ref{propg}). For $\chr\k=2$ we will also need to demonstrate that a generic tangent hyperplane to $C$ is only tangent at one point. We assume henceforth that $\chr\k=2$ and no component of $C'$ is contained in a plane.
Next we will need a few auxiliary statements.

\begin{lem}\label{lemstrange} Let $D\ss\P^n$ be a quasireflexive curve and assume that no component of $D$ is a conic. Then $D$ is not strange. \end{lem}
\begin{proof} Assume to the contrary that for some component $D_i$ of $D$ all its tangent lines meet at a point $S$. Then any hyperplane containing $S$ is tangent at every point of $H\cap D_i$. On the other hand since $D$ is quasireflexive, there exists a tangent hyperplane $H$ to $D_i$ which is tangent at a single point $P$ with $I(P,H.D_i)=2$. Consequently we have $H\cap D_i=\{P\}$. Now by Bezout's theorem we have $\deg D_i=I(P,H.D_i)=2$ and $D_i$ is a conic, contrary to assumption.\end{proof}

\begin{lem}\label{lem0} Let $D\ss\P^n$ be an irreducible curve not contained in a plane. Assume that for all smooth $P,Q\in D$ the tangents $L_P,L_Q$ are coplanar. Then $D$ is strange.\end{lem}

\begin{proof} This is stated for $M=3$ in \cite{Har77}*{Proposition 3.8} but the proof there works in any dimension.\end{proof}

\begin{lem}\label{lem1} Let $D\ss\P^n,n\ge 2$ be an irreducible curve such that for some component $D_i$ a generic $H\in D_i^*$ is tangent to $D$ at more than one point. Then one of the following holds:
\begin{enumerate}
\item[(i)] For a generic $P\in D_i$ there exists $Q\in D,Q\neq P$ such that $L_P=L_Q$.
\item[(ii)] For some (possibly identical) component $D_j$ and generic $P\in D_i,Q\in D_j$ the tangents $L_P,L_Q$ are coplanar.
\end{enumerate}
Further, if $D$ is irreducible and not strange then (i) holds.
\end{lem}

\begin{proof} We may assume that $D$ is not contained in a hyperplane, otherwise we can restrict the ambient space to this hyperplane. If $n=2$ then our assumption is equivalent to (i), so we may assume $n\ge 3$ and then $D$ is not contained in a plane. By assumption a generic $H\in D_i^*$ contains $L_P,L_Q$ for some $P\in D_i,Q\in D,P\neq Q$.
Consider the variety $$V=\{(H,P,Q)\in D_i^*\times D_i\times D|L_P,L_Q\ss H,P\neq Q\}$$ and the
projection map $\al:V\to D_i\times D$. By assumption the projection $V\to D_i^*:(H,P,Q)\mapsto H$ is dominant, so $\dim V=n-1$.

For smooth $P\in D_i,Q\in D$ with $P\neq Q$ we have
$$\dim\al^{-1}(P,Q)=\left[\ba{ll}n-2,&L_P=L_Q,\\n-3,&L_P,L_Q\mbox{ are coplanar},\\
n-4,& L_P\cap L_Q=\emptyset.\ea\right.$$
If property (i) doesn't hold then $\dim\al^{-1}(P,Q)=n-2$ only for finitely many pairs $P\in D_i,Q\in D$. Consequently the only way we could have $\dim V=n-1$ is if for some component $D_j$ and every smooth $P\in D_i,Q\in D_j$ the tangents $L_P,L_Q$ are coplanar, so (ii) holds.

The last assertion of the lemma follows from Lemma \ref{lem0} (by the assumptions we made in the beginning of the proof the curve $D$ is not contained in a plane).
\end{proof}

Now we go back to our curves $C,C'$. By Lemma \ref{lemstrange} the curve $C'$ is not strange (by assumption no component of $C'$ can be a conic since it does not lie on a plane) and therefore $C$ is not strange. Indeed if the all tangent lines to some component $C_i$ meet at a point $S$ we must have $S\neq O$ (since $O$ is not a strange point for $C$) and then the tangent lines to $C_i'$ meet at $\pi(S)$.

Next we claim that $C$ cannot satisfy assertion (i) in Lemma \ref{lem1}. Indeed assume that for a generic point $P\in C_i$ there exists a $Q\in C,Q\neq P$ such that $L_P=L_Q$. The point $O$ can lie on $L_P$ only for finitely many $P$, so generically $\pi(P)\neq\pi(Q)$ but $L_{\pi(P)}=L_{\pi(Q)}$. Thus the generic tangent line to $C_i'$ is tangent to $C'$ at another point. This implies the same for generic tangent hyperplanes, contrary to assumption.

Now assume by way of contradiction that a generic tangent hyperplane to some component $C_i$ is tangent to $C$ at more than one point. Since assertion (i) in Lemma \ref{lem1} cannot hold for $C$, it must satisfy assertion (ii), i.e. there is a component $C_j$ (possibly $j=i$) such that for every smooth $P\in C_i,Q\in C_j$ the lines $L_P,L_Q$ are coplanar. By projection the same is true for $C_i',C_j'$. If $C_i'=C_j'$ this is impossible by Lemma \ref{lem0} since by assumption $C_i'$ is not contained in a plane and we have shown that it is not strange.

If $C_i'\neq C_j'$ we will show directly that there is a hyperplane $H\ss\P^n$ tangent to $C_i$ but not to $C_j$. If it happens that $O\in C$ and is a singular point then it can have several (but finitely many) tangent lines $T_1,\ldots,T_n$. Since $C'$ is quasireflexive we may find a hyperplane $h\ss\P^{n-1}$ such than $h\in C_i^*\sm C_j^*$ and $h$ does not contain any of the points $\pi(T_k)$ or any points of $C$ over which $\pi$ is ramified. Then $H=\pi^{-1}(h)$ is tangent to $C_i$ but not to $C_j$. This concludes the proof.

{\bf Remark.} If $\chr\k=2$ the requirement that $C'$ has no component which lies on a plane cannot be dropped. For example consider the curve $C$ with affine model $$\{(t:t^2:t^4)|t\in\k\}\in\A^3$$ and the projection from $(0:0:0:1)$ which acts by $(t,t^2,t^4)\mapsto (t,t^2)$. Its image $C'=\pi(C)$ is a conic which is quasireflexive, but $C$ is not. Indeed the generic tangent hyperplane $a_0+a_1x_1+a_2x_2+x_3=0$ satisfies $a_1=0$ and is tangent at two distinct points, since $a_0+a_1t+a_2t^2+a_4t^4=0$ has a double root iff $a_1=0$, in which case it has (generically) two double roots (compare with the proof of Lemma \ref{propnormal}). We do not know whether in all such examples one of the components of $C'$ must be a conic.

\bibliography{mybib}
\bibliographystyle{amsrefs}

\end{document}